\documentclass[11pt]{amsart}  
\usepackage[top=1.5in, bottom=1.5in, left=1.4in, right=1.4in]{geometry}    
\usepackage{graphicx}
\usepackage{amssymb}
\usepackage{epstopdf}
\usepackage{amsthm}
\title{Infinite Reduction of Divisors on Metric Graphs}
\author{Spencer Backman}

\newtheorem{theorem}{Theorem}
\newtheorem{lemma}{Lemma}

\begin{document}

\begin{abstract}
We demonstrate that the greedy algorithm for reduction of divisors on metric graphs need not terminate by modeling the Euclidean algorithm in this context.  We observe that any infinite reduction has a well defined limit allowing us to treat the greedy reduction algorithm as a transfinite algorithm and to analyze its running time via ordinal numbers.   We provide lower and upper bounds which establish a worst case running time of $\omega^{\Theta({\rm deg}(D))}$.

\end{abstract}

%\begin{keyword}
%chip-firing \sep Dhar's algorithm \sep  Euclidean algorithm \sep greedy algorithm  \sep metric graph \sep ordinal numbers \sep tropical curves 
%\end{keyword}

\maketitle
%\section{}
%\subsection{}

\section{Introduction}
Chip-firing on graphs has been studied in various contexts for over 20 years.  The theory has found new applications in the recent work of Baker and Norine \cite{Baker} who utilized chip-firing to prove a Riemann-Roch theorem for graphs analogous to the classical statement for curves.  Their work demonstrates that chip-firing gives an extremely useful combinatorial language for investigating linear equivalence of divisors on graphs, a topic which was previously investigated from a more algebraic and geometric perspective \cite{Bacher, Lorenzini1, Lorenzini2}.   Gathmann and Kerber \cite{Gathmann}, and independently Mikhalkin and Zharkov \cite{Mikhalkin}, proved a Riemann-Roch theorem for tropical curves.  The approach of Gathmann and Kerber was to establish the tropical Riemann-Roch theorem as a limit of Baker and Norine's result for graphs under subdivision of edges.  Hladky, Kral, and Norine \cite{Norin} then showed that this theorem may be proven in an elementary way by studying the combinatorics of chip-firing on abstract tropical curves, i.e., metric graphs.  Several papers have pursued this approach further along with other consequences for the theory of linear equivalence of divisors on tropical curves  \cite{Amini} \cite{Cools} \cite{Musiker} \cite{Luo}.

The central combinatorial objects in this study, for both graphs and tropical curves, are the so-called $q$-reduced divisors (known elsewhere in the literature as superstable configurations or $G$-parking functions).  A $q$-reduced divisor is a special representative from the class of divisors linearly equivalent to a given divisor.  There is an algorithmic method for obtaining the unique $q$-reduced divisor consisting of two parts.  In this paper, we investigate the second, more subtle part of this process known as reduction.  We offer a short proof that Luo's metric variant of Dhar's reduction algorithm terminates after a finite number of iterations.  We then investigate the greedy reduction algorithm, which in the graphical case is known to succeed.  We show that the Euclidean algorithm may be modeled by the greedy reduction of divisors on metric graphs.  By taking a pair of incommensurable numbers as input, we obtain a run of the greedy reduction algorithm which does not terminate.  

After observing that any infinite reduction has a well-defined limit,  we analyze the running time of the greedy algorithm via ordinal numbers.  We provide nearly matching upper and lower bounds which establish that the worst case running time of $\omega^{\Theta({\rm deg}(D))}$.  The lower bound is obtained by gluing $n$ copies of the Euclidean algorithm example together and ordering the firings lexicographically.  The upper bound of $\omega^{{\rm deg}(D)}$ is provided by an inductive argument.

\section{Metric Chip-Firing and Reduced Divisors}

       A metric graph $\Gamma$ is a metric space which can be obtained from an edge weighted graph $G$ by viewing each edge with weight $w_{i,j}$ as being isomteric to an inteval of length $w_{i,j}$.  Each point interior to an edge has a neighborhood homeomorphic to an open interval and each vertex has a small neighborhood homeomorphic to a star.  The {\it degree} of a point $p \in \Gamma$ is the number of tangent directions at $p$.  A vertex is called a {\it combinatorial vertex} if it has degree other than 2.
       
      This paper concerns certain combinatorial aspects of chip-firing on metric graphs, so we will take a rather concrete working definition of chip-firing.  For completeness sake, we begin with a slightly more abstract definition.  Fix a metric graph $\Gamma$ and a parameterization of the edges of $\Gamma$.  Let $f$ be a piecewise affine function with integer slopes on $\Gamma$.  We define the Laplacian operator $Q$ applied to $f$ at a point to be the sum of the slopes of the function as we approach $p$ along each of the tangent directions at $p$.  We note that $Q(f)(p)=0$ if $f$ is differentiable at $p$.  We define a {\it divisor} $D$ on $\Gamma$ to be a formal sum of points from $\Gamma$ with integer coefficients, all but a finite number of which are zero.  We say that $D$ has $D(p)$ {\it chips} at $p$.  Given some divisor $D$ on $\Gamma$, we define the chip-firing operation $f$ applied to $D$ to be $D-Q(f)$.  We say that two divisors are linearly equivalent if they differ by some chip-firing move.   A divisor E is said to be {\it effective} if it has a nonnegative number of chips at each point.

      We now give the definition of chip-firing on metric graphs which will be used for the remainder of the paper.   Let $X$ and $Y$ be two disjoint open connected subsets of $\Gamma$ such that the $\Gamma \setminus (X\cup Y) =  Z$ is isometric to a disjoint collection of closed intervals of length $\epsilon$.  Note that the set $Z$ defines a minimal cut in $\Gamma$.  Now, we define the divisor $Q(f)$ as the divisor which is negative one at the end points of these intervals on the boundary of $X$ and positive one at the endpoints on the boundary of $Y$.  One may intuitively understand this divisor as pushing a chip along each edge in this cut a fixed distance $\epsilon$.  We take this to be the basic type of chip-firing move and call $\epsilon$ the {\it length} of the firing.  Note that the chip-firing divisor is of the form $Q(f)$ where $f$ is the piecewise affine function with integer slopes which is 0 on $X$, $\epsilon$ on $Y$, and has slope 1 on the each  open interval in $Z$.  We write $\epsilon (f)$ for the length of the firing $f$.  As is noted in \cite{Baker3}, any piecewise affine function with integer slopes can be expressed as a finite sum of the functions just described, so we will not sacrifice any generality by restricting our definition of chip-firing to be basic chip-firing moves.      
      
      A {\it $q$-reduced divisor} is a divisor which is nonnegative at each point other than $q \in \Gamma$, such that any firing $Q(f)$ which pushes chips toward $q$ causes some point to go into debt.  It is proven in \cite{Norin} that given any divisor $D$ on a metric graph $\Gamma$, there exists a unique $q$-reduced divisor $\nu$ which can be reached from $D$ by a sequence of chip-firing moves.   Moreover, there exists an effective divisor $E$ equivalent to $D$ if and only if $\nu$ is effective.  An algorithmic way of obtaining such a divisor was described by Luo \cite{Luo}.  His method is to first bring every point other than $q$ out of debt by some sequence of chip-firing moves.  Once we have obtained such a configuration, we may perform firings which push chips back toward $q$ without causing any vertex to go into debt.  We call this second part of the process {\it reduction}.  Luo's method for reducing a divisor is to use a generalization of Dhar's burning algorithm originally investigated in the study of the sandpile model.
            
      Dhar's burning algorithm may be described in the following informal way:  Let $D$ be a divisor which is nonnegative at every point of $\Gamma$ other than $q$.  Place $D(p)$ firefighters at each point $p$ other that $q$.  Light a fire at $q$ and let the fire spread through $\Gamma$ along the edges.  Every time the fire reaches a firefighter, it stops.   If the fire approaches a point from more directions than there are firefighters present, these firefighters are overpowered and the fire continues to spread through the $\Gamma$.  It is not hard to check that a divisor is $q$-reduced if and only if the fire burns through the entirety of $\Gamma$.
      
          Let $D$ be nonneqative at all points other than $q$.  We say that a firing $f$ is {\it legal } if $D-Q(f)$ is also nonnegative at all points other than $q$.  A firing $f$ is a {\it maximal legal firing} for $D$ if the legal firing $f$ is taken to have maximum associated length, i.e., the firing describes a push of chips along a cut so that at least one of the chips hits a combinatorial vertex and therefore cannot be pushed any further without choosing a different cut.  Every time a fire is prevented from burning through $\Gamma$, the collection of chips, i.e., firefighters which stop the fire define a maximal legal firing  towards $q$ which does not cause any point to go into debt.  Luo showed that if we take a divisor which is non-negative away from $q$ and reduce according to the maximal legal firings obtained from this algorithm, we will obtain a reduced divisor after a finite number of steps.  We offer the following new short proof that this process terminates.
      
\begin{theorem}[Luo \cite{Luo}, Theorem 2.14]
 Let $\Gamma$ be a metric graph and $D$ be a divisor on $\Gamma$ such that $D(p)\geq0$ for all $p \in \Gamma$ with $p\neq q$.  The reduction of $D$ with respect to $q$ by maximal legal firings obtained from Dhar's algorithm terminates after a finite number of iterations.
\end{theorem}

\begin{proof}
Let $\#E(\Gamma)$ be the number of edges in $\Gamma$ and without loss generality take $D$ to be a divisor which has 0 total chips.  We proceed by double induction on $\#E(\Gamma)$ and $- D(q)$, i.e., induction on $k =\#E(\Gamma)-D(q)$.  We will restrict our attention to the set $S$ of edges adjacent to $q$.  First observe that for any edge $e$ in $S$ which has a chip, the closest chip $p$ in $e$ to $q$ will never leave $e$, although chips in $e$ further away from $q$ may.  This is because when the fire burns from $q$, it will always reach $p$.  If $p$ does eventually reach $q$, then we may induct on $k$.  Also, if every edge in $S$ contains a chip, on the next iteration of the algorithm, the fire will burn from $q$ and be stopped by precisely these edges.  Hence, these chips will be fired towards $q$, which will receive a chip and we may again induct on $k$.  Therefore, we may assume that there exists an edge in $S$ which will never receive a chip.   We can contract this edge and again induct on $k$. 
\end{proof}

\section{Infinite Greedy Reduction}

    We can always reduce by performing any maximal legal firings we wish.  If this process terminates, we know by uniqueness that we have reached the $q$-reduced divisor equivalent to $D$.  In the case of discrete graphs, it is clear that this process will terminate. We begin by providing a negative answer to Matthew Baker and Ye Luo's question of whether the greedy reduction algorithm for metric graphs also terminates after a finite number of iterations \cite{Baker3.5}.  To this effect, we will provide an example which demonstrates that we can model the Euclidean algorithm in this context and therefore, by taking our input to be a pair of incommensurable numbers, obtain a run of the greedy reduction algorithm which does not terminate in finite time.

We define a {\it greedy reduction} of a divisor $D = D_0$ to be a sequence of divisors $D_i$ such that $D_i = D_{i-1} - Q (f_{i-1})$, where $Q(f_{i-1})$ is a maximal legal firing for $D_{i-1}$.

\begin{theorem}
 There exists a metric graph $\Gamma$, a divisor $D$ on $\Gamma$ such that $D(p)\geq0$ for all $p \in \Gamma$ with $p\neq q$, and a greedy reduction of $D$ with respect to $q$ which does not terminate after a finite number of steps.
\end{theorem}

 \centerline {\includegraphics[scale=0.65]{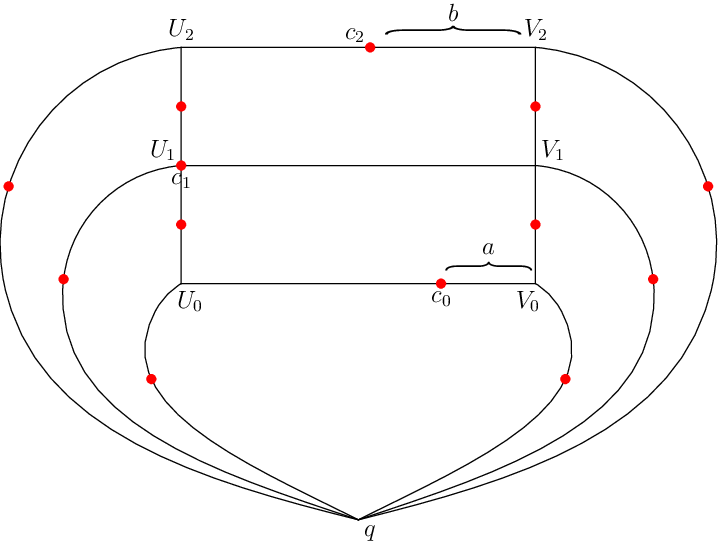}}
\

    \centerline{   \includegraphics[scale=0.65]{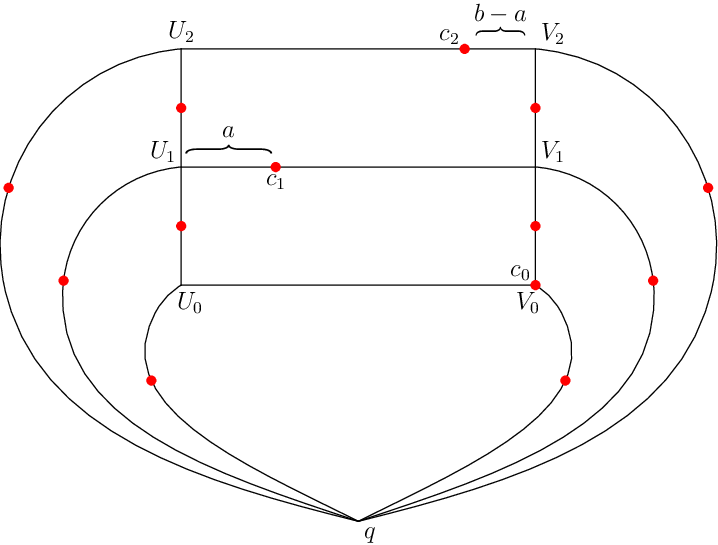}}
       
\ 

    \centerline{   \includegraphics[scale=0.65]{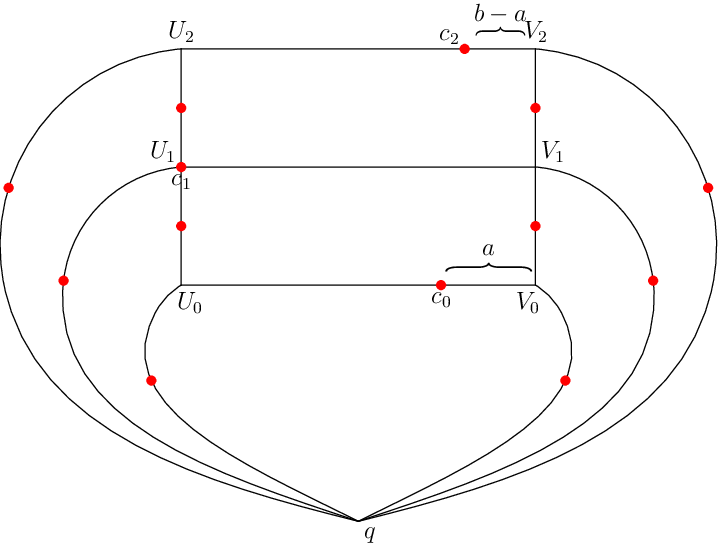}}
       
  \centerline {Figure 1.  A divisor on a metric graph with infinite greedy }
  \centerline {reduction via firings which model the Euclidean algorithm.}

\

\begin{proof}

       We now present an example which demonstrates that the greedy reduction algorithm may not terminate after a finite number of iterations.  We refer to Figure 1 which illustrates a certain divisor $D$ on a metric graph $\Gamma$.  We will take all of the edge lengths to be sufficiently large.  What is meant by sufficiently large will become clear after we have completed the proof.  We take $D$ to have chips (with labels for the clarity's sake) $c_0, c_1,$ and $c_2$ with $c_1$ at $u_1$, $c_0$ at distance $a$ from $v_0$ on $(u_0, v_0)$ and $c_2$ at distance $b$ from $v_2$ on $(u_2, v_2)$ with $a<b$.  We take $D$ to have a chip at the midpoint of every other edge in $\Gamma$.  It is not important that chips be at the midpoints, only that they be sufficiently far from both endpoints.  The idea of the example is to show that given $D$, we can perform the subtraction of $a$ from $b$ without changing the rest of the divisor much.  We may then perform the Euclidean algorithm on inputs $a$ and $b$.  By taking $a$ and $b$ so that ${a\over b} \notin \mathbb{Q}$ it follows (after verifying the convergence of a certain series) that we can obtain a run of the greedy reduction algorithm which does not terminate.  We now describe the pair of firings which allows us to subtract $a$ from $b$.
       
       Firing 1:  We would like to perform a maximal legal chip-firing move towards $q$ which will push chips $c_0$, $c_1$, and $c_2$ length $a$ toward $v_0$, $v_1$, and $v_2$ so that $c_0$ hits $v_0$.  We can achieve this by taking a firing corresponding to the cut $(X,Y) = (\{u_0, u_1, u_2\}, \{q, v_0, v_1, v_2 \})$.  Given this cut, we can push $c_0, c_1$, and $c_2$ as described and extend this to a maximal legal firing towards $q$ by pushing the chips interior to the edges $(u_0,q), (u_1,q)$ and $(u_2,q)$ distance $a$ towards $q$.
       
       Firing 2:  Now, we would like to perform a maximal legal chip-firing move towards $q$ which will push $c_0$ and $c_1$ distance $a$ back towards $u_0$ and $u_1$ respectively so that $c_1$ reaches $u_1$.  As in Firing 1, we can achieve this by the taking the cut corresponding to the partition $$(X,Y) = (\{v_0,v_1\},\linebreak  \{v_2, u_0, u_1, u_2, q \} )$$ and pushing chips in each of the other edges of this cut length $a$ towards $\{v_2, u_0, u_1, u_2, q \}$.  
       
       By ignoring the position of all of the chips other than $c_0, c_1$, and $c_2$, we observe that we have returned to the orginal divisor with $b$ replaced by $b-a$, so we have subtracted $a$ from $b$.  We can now perform the Euclidean algorithm by subtracting $a$, $n$ times from $b$ so that $0 \leq b-na <a$.  By the symmetry of the construction, we may now reverse the roles of $c_0$ and $c_2$ and repeat.  The one subtlety here is that we need to be sure that none of the other chips in the metric graph eventually reach either of the endpoints of the edge they are contained in, otherwise we might not be able to perform the firings described above.  This is why we take the chips to be at the midpoints of sufficiently long edges.  If $a$ and $b$ are such that ${a\over b} \notin \mathbb{Q}$, this process will never terminate, but the series of length of the firings will converge, and we can take the lengths of the edges to be the twice this series of the lengths corresponding to the firings performed.  It remains to prove that the corresponding series of lengths converges.  To this end, we will assign some notation to the quantities appearing in the Euclidean algorithm.  Given two numbers $a_i$ and $b_i$ with $0< a_i < b_i$, we define $b_{i+1}=a_i$  and $a_{i+1} = b_i -n_ia_i$ with $n_i \in \mathbb{N}$ and $0 \leq b_i -n_i a_i < b_i$.  Letting $l_i = n_ia_i$, it needs to be shown that $\sum_{i \geq 0} l_i$ converges.  We claim that taking $a=a_0$ and $b=b_0$, $\sum_{i \geq 0} l_i \leq 4 b $.  This follows from the simple observations that $l_{i+1} \leq l_i$ and $l_{i+2} < {1\over 2} l_i$, which allow us to conclude that $ \sum_{i \geq 0} l_i$ is bounded geometrically and the claim follows.
       
\end{proof}

\section{Running Time Analysis via Ordinal Numbers}       
    
             We now prove than any reduction of a divisor which does not terminate has a well-defined limit.  This will allow us to interpret the greedy algorithm as a transfinite algorithm and to analyze its running time in the language of ordinal numbers.  We first prove that for any infinite reduction, the sum of the lengths of the firings must converge.   
       
 \begin{lemma}
 Let $\Gamma$ be a metric graph, $D$ a divisor on $\Gamma$ such that $D(p)\geq 0 $ for all $p \neq q$, and let $f_i$ be an infinite sequence of maximal legal firings reducing $D$.  Then the series $\sum_i f_i$ converges and the greedy reduction of $D$ has a well defined limit.
  \end{lemma}
  
  \begin{proof} 
  
   Let $l(p) = \sum_{i=0}^{\infty} f_i(p)$ for $p \in \Gamma$.  We now take $v$ and $v'$ to be combinatorial vertices.  If $l(v)$ is finite then $l(v')$ is finite -- if this were not the case, it would mean that $v'$ sent an infinite number of chips towards $v$ which were never able to leave the set of edges incident to $v$ and so we would have an infinite number of chips clustered around $v$, a clear contradiction.  Take some $v$ adjacent to $q$.  Clearly $l(v)$ is finite, otherwise $v$ will send an infinite number of chips to $q$.  By the connectedness of the metric graph, it follows that $l(v)$ is finite for each combinatorial vertex, hence $l(v)$ is finite for each $v$ and it follows that $\sum_i f_i$ converges. 
   
   We now show that $\sum_{i=0}^{\infty} \epsilon (f_i)$ converges.  Because $\Gamma$ is compact and the sum of $f_i$ is convergent, we see that the limit of $\sum_{i=0}^{\infty} \int_{\Gamma} f_i$ is well defined.  Moreover, $\int_{\Gamma} f_i \geq \epsilon( f_i)m(\Gamma)$, where $m(\Gamma)$ is the sum of the lengths of the edges of $\Gamma$, as $f$ is a nonnegative piecewise affine function with slopes $\pm1$, therefore $\sum_{i=0}^{\infty} \epsilon (f_i)$ converges.

Label the chips in $D$ arbitrarily.  For each passage from $D_i$ to $D_{i+1}$, a given chip $c$ either stays fixed or travels $\epsilon (f_i)$.  The series of these lengths which $c$ travels must converge because it is an increasing sequence bounded above by $\sum_{i=0}^{\infty} \epsilon (f_i)$.  Therefore as we follow the path which this chip traces out, we see that it must have a well defined limit.  Hence the limit of the greedy reduction has a well defined limit.
  \end{proof}

            It is now natural, given an infinite greedy reduction,  to pass to the limit and begin the process again.  We will analyze the running time of the greedy reduction algorithm in terms of ordinal numbers.  For an introduction to ordinal numbers, we refer the reader to \cite{Hrbacek}.  We remark that in what follows, we will not use any advanced properties of ordinal numbers, rather they serve as a bookkeeping tool for rigorously investigating the question of how long it takes for the greedy reduction of a divisor to terminate.   It has been proven \cite{Baker4}\cite{Mikhalkin} that any convergent sum of basic chip-firing moves is itself a finite sum of basic chip-firing operations, so we can be confident that in passing to the limit of a chip-firing process, we never leave the class of divisors linearly equivalent to the one we started with.  In what follows, we would like to emphasize that $\omega ^n$ is not $n \omega$, that is, even informally we should not think of $\omega ^n$ as $n$ copies of $\omega$ concatenated, rather we should consider this quantity as a nest of $\omega$'s of depth $n$.

 \begin{lemma}\label{lower}
        For every $n \in \mathbb{N}$, there exists a metric graph $\Gamma$ and a divisor $D$ on $\Gamma$ with $D(p)\geq 0$ for all $p \in \Gamma$ with $p \neq q$ such that the greedy reduction of $D$ with respect to $q$ takes time at least $ \omega^n$.
\end{lemma}

\begin{proof}
         This is again a proof by construction.  The idea is that by ``gluing together" $n$ copies of the previous Euclidean example we can obtain running time $\omega^n$.  See Figure 2 for an illustration of a piece of the example which will allow us to obtain running time $\omega^2$.  Missing from the figure are the edges $(u_i,q)$ and $(v_i,q)$ for all $i$.  Once again, we imagine that all of the edges are sufficiently long and that there are chips at each of the midpoints of the edges not drawn.  The idea is that we can run the example described previously for the bottom 3 rows (letting $(u_0,v_0,c_0)$ and $(u_1,v_1,c_1)$ switch roles).  In the limit, $c_0, c_1,$ and $c_2$ move to $u_0, v_1$ and $v_2$ respectively, after which we may use the top two rows to ``recharge" $c_1$ and $c_2$.  This recharging is achieved by the following two firings, which are illustrated in Figure 2:
         
         Firing 1:  We would like to push $c_0$ and $c_3$ distance $a$ towards $v_0$ and $v_3$ respectively so that $c_3$ hits $v_3$.  We can achieve this by taking a maximal legal firing corresponding to the cut $$(X,Y) = (\{u_0,u_3\},\{q,u_1,u_2,u_4,v_0,v_1,v_2,v_3,v_4\}).$$
         
         Firing 2:  We would like to push $c_0, c_1$ and $c_3$ distance $a$ towards $v_0,v_1$ and $v_3$ respectively so that $c_0$ hits $v_0$.  We can achieve this by taking a maximal legal firing corresponding to the cut $$(X,Y) = (\{v_0, v_1, v_3\}, \linebreak \{q, u_0, u_1, u_2, u_3, u_4, v_2, v_4\}).$$
         
         The figure shows how we can use $c_3$ and $c_0$ to recharge $c_1$.  We can then perform a similar pair of firings using $c_4$ and $c_0$ to recharge $c_2$.  We then iterate the process.  The one subtle point is that we again need convergence of the double series of lengths coming from the firings.  In order to do this, we should perform one step of the Euclidean algorithm with $c_0, c_3$ and $c_4$ after taking a limit of the bottom three rows and before recharging $c_1$ and $c_2$.  A simple calculation shows that this minor adjustment ensures convergence.  We leave the extension to $n$ copies of the Euclidean example as an exercise for the reader.  In order to ensure convergence of the associated $n$ nested series of lengths, we should order the firings on the product of the n copies of the Euclidean example lexicographically.

       \centerline{\includegraphics[scale=0.65]{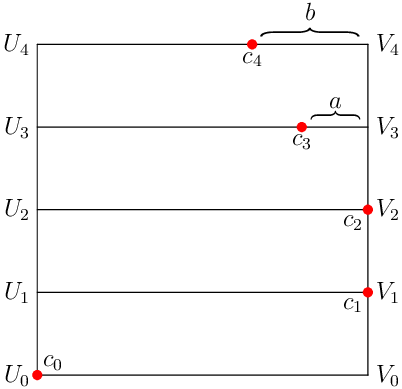}}

\centerline{       \includegraphics[scale=0.65]{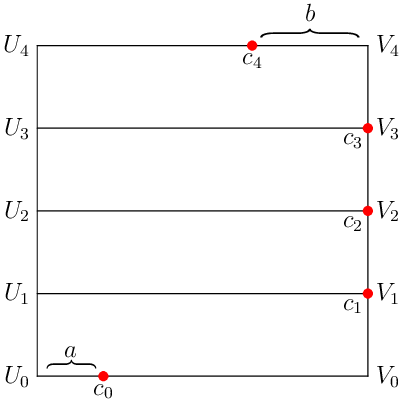}}
       
        \centerline{     \includegraphics[scale=0.65]{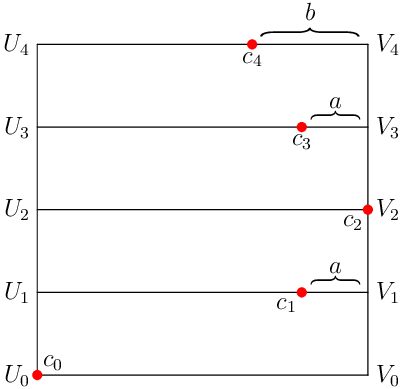}}
        
        \

\centerline{Figure 2.  Two copies of the Euclidean example glued}
\centerline { together and a recharging move achieved by two firings.}

\end{proof}

        We now show that in some sense, the previous example is worst possible.  The previous example shows that for any ordinal number $\alpha < \omega^{\omega}$, there exists a divisor $D$ with a greedy reduction which takes more than $\alpha$ steps.  While it is more or less obvious that we cannot have a greedy reduction with an uncountable number of iterations, $\omega^{\omega}$ is still a countable ordinal and so {\it a priori} there might exist a divisor with a greedy reduction which takes $\omega^{\omega}$ steps.  The following result shows that this cannot occur.       
                
\begin{lemma}\label{upper}
Let $\Gamma$ be a metric graph and $D$ be a divisor on $\Gamma$ such that $D(p)\geq 0$ for all $p \in \Gamma$ with $p \neq q$.  Any greedy reduction of $D$ with respect to $q$ takes at most $\omega^{deg(D)}$ steps.  
\end{lemma}

\begin{proof}        

	Insight into this claim can be derived from inspection of the Euclidean example.  As was noted previously, when we pass to the limit of this reduction process, $c_0$, $c_1$, and $c_2$ approach the combinatorial vertices $u_0$, $v_1$, and $v_2$ respectively.  We claim that more generally at step $\omega^{n-1}$ of any greedy reduction, there must be at least $n$ chips present at the combinatorial vertices.  Eventually, all of the chips will be at combinatorial vertices.  During the next firing, some chip will traverse an edge from one combinatorial vertex to another.  Thus the length of the firing will be at least the minimum the edge lengths.  This cannot happen an infinite number of times otherwise the sum of the lengths will diverge contradicting Lemma 1.  This will in turn give an upper bound on the running time of the greedy reduction algorithm of  $\omega^{deg(D)}$.
        
        We will proceed by induction.  Because we are performing maximal legal firings, we always have a chip at a combinatorial vertex, e.g., one of the chips which arrived at a combinatorial vertex after the previous firing.  We take this to be the base case of the claim.  For the inductive step, assume that there are at least $n$ chips at the combinatorial vertices at time $\omega^{n-1}$.  For each step $k\omega^{n-1}$, we can associate a set of chips $S_k$ which are present at the combinatorial vertices at this time.  Let $A$ be some set of chips which is equal to $S_k$ for infinitely many $k$.  At time $\omega^{n}$, the set of chips $A$ will lie at combinatorial vertices.  Moreover, if there exists some set of chips $B \neq A$ which is equal to $S_k$ for infinitely many $k$, then the union $A\cup B$ will be present at combinatorial vertices at time $\omega^{n+1}$, and we will have proved the claim, therefore we may assume that there exists a unique $A$ equal to $S_k$ for infinitely many $k$.  At time $k\omega^n +1$ some chip $c_k$ must reach a combinatorial vertex.  Observe that the chip $c_k \in S_k$ for only finitely many $k$, otherwise the nested series of lengths will diverge.  Therefore there exist some chip $c = c_k$ for infinitely many $k$ such that $c \notin S_k$.  We conclude that $A \cup {c}$ are living at combinatorial vertices in the limit at time $\omega^{n}$ thus completing the proof.
        
\end{proof}

 We note that although we are working with ordinal numbers, the previous argument employed only finite induction, not transfinite induction.  Recall that a function is $f$ is $\Theta(g)$ if there exist positive constants $c_1$ and $c_2$ such that asymptotically   $c_1 g \leq f \leq c_2 g$.  Now we are ready to state the main theorem of this section.
 
 \begin{theorem}
 The worst case running time of the greedy reduction algorithm for divisors on metric graphs is $\omega^{\Theta( { \rm deg}(D))}$.
 \end{theorem}
 
 \begin{proof}
 This follows directly from Lemma \ref{lower} and Lemma \ref{upper}.
 \end{proof}
 
   We conclude with a question posed to the author by Sergey Norine.  The previous bound is a function of the degree of the divisor, but not the metric graph $\Gamma$.  It would be nice to have an upperbound on the running time of an arbitrary greedy reduction on $\Gamma$ in terms of the structure of $\Gamma$.  Precisely, does there exist a bound of the form $\omega^{f(g)}$ where $f$ is a polynomial in $g$, the {\it cyclomatic genus} of $\Gamma$?  Can we take $f$ to be linear?
 
 \

 \centerline{Acknowledgements:}

 Many thanks to Ye Luo and Matthew Baker for posing the original question which prompted this paper and for productive discussions.  Additional thanks for Farbod Shokrieh and Sergey Norine for helpful comments and conversations.  Thanks to Maria Gillespie for help with the figures, and a special thanks to Pedro Rangel with whom discussions on ordinal numbers inspired the last section of this paper.  Additional thanks to the Bellairs Research Institute where some of this work was completed during the Tropical and Non-Archimedean Geometry Workshop.

\bibliographystyle{model1a-num-names}
\bibliography{<your-bib-database>}

\end{document}